\documentclass[12pt]{amsart}

\usepackage[utf8]{inputenc}
\usepackage{pdfsync}
\usepackage[margin=1 in]{geometry}

\usepackage[sort&compress,square,numbers]{natbib}
\usepackage{graphicx,amssymb,amstext,amsmath,wrapfig,url,bm}
\usepackage[colorlinks]{hyperref}
\usepackage[hang]{subfigure}
\usepackage{pxfonts}

\newcommand{\card}[1]{\ensuremath{\left\vert #1 \right\vert}}

\renewcommand{\vec}[1]{\mathbf{#1}}

\def\complaint#1{}
\def\withcomplaints{

\newcounter{mycomplaints}
\def\complaint##1{\refstepcounter{mycomplaints}%
\ifhmode%
\unskip%
{\dimen1=\baselineskip \divide\dimen1 by 2 %
\raise\dimen1\llap{\tiny -\themycomplaints-}}\fi%
\marginpar{\tiny [\themycomplaints]: ##1}}%
}

\def\lemmaD#1{
\begin{lemma}
\label{lem:#1}
}

\newcommand{\theoremD}[1]{
\begin{theorem}
\label{theorem:#1}
}

\newcommand{\factD}[1]{
\begin{fact}
\label{fact:#1}
}

\newcommand{\corD}[1]{
\begin{cor}
\label{cor:#1}
}

\usepackage{amsmath,amsthm}

\usepackage{theomac}

\newcommand{\lemlab}[1]{\label{lemma:#1}}
\newcommand{\theolab}[1]{\label{theo:#1}}

\newcommand{\eqlab}[1]{\label{eq:#1}}

\newcommand{\lemref}[1]{Lemma \ref{lemma:#1}}

\newcommand{\theoref}[1]{Theorem \ref{theo:#1}}

\renewcommand{\eqref}[1]{(\ref{eq:#1})}

\newtheoremWithMacro{theorem}{Theorem}

\newtheorem{lemma}[theorem]{Lemma}
\newtheoremWithMacro{prop}[theorem]{Proposition}
\newtheorem{cor}[theorem]{Corollary}

\begin{document}

\title{Lines induced by bichromatic point sets}
\author{Louis Theran}
\address{Department of Mathematics \\
Temple University}
\email{theran@temple.edu}
\urladdr{http://math.temple.edu/~theran/}

\begin{abstract}
An important theorem of Beck says that any point set in the Euclidean plane is either ``nearly general position'' or
``nearly collinear'': there is a constant $C>0$ such that, given $n$ points in $\mathbb{E}^2$ with at most $r$
of them collinear, the number of lines induced by the points is at least $Cr(n-r)$.

Recent work of Gutkin-Rams on billiards orbits requires the following elaboration of Beck's Theorem
to bichromatic point sets: there is a constant $C>0$ such that, given $n$ red points and $n$ blue points
in $\mathbb{E}^2$ with at most $r$ of them collinear, the number of lines spanning at least one point of each
color is at least $Cr(2n-r)$.
\end{abstract}

\maketitle

\section{Introduction}
Let $\vec p$ be a set of $n$ points in the Euclidean plane $\mathbb{E}^2$
and let $\mathcal{L}(\vec p)$ be the set of lines induced by $\vec p$.  A
line $\ell\in \mathcal{L}(\vec p)$ is $k$-rich if it is incident on
at least $k$ points of $\vec p$.  A well-known theorem of Beck relates
the size of $\mathcal{L}(\vec p)$ and the maximum richness.
\begin{theorem}[\beckthm][{\bf Beck's Induced Lines Theorem \cite{B83}}]\theolab{beck}
Let $\vec p$ be a set of $n$ points in $\mathbb{E}^2$, and let $r$ be the maximum richness
of any line in $\mathcal{L}(\vec p)$.  Then $\card{\mathcal{L}(\vec p)}\gg r(n-r)$.
\end{theorem}
Here $f(n)\gg g(n)$ means that $f(n)\ge Cg(n)$, for an absolute constant $C>0$.

In this note, we give an elaboration (using pretty much the same arguments)
of Beck's Theorem to bichromatic point
sets, which arises in relation to the work of Gutkin and Rams \cite{GR10} on the dynamics
of billiard orbits.  Let $\vec p$ be a set of $n$ red points and let $\vec q$
be a set of $n$ blue points with all points distinct (for a total of $2n$).
We define $\vec p\cup \vec q$ to be the \emph{bichromatic
point set $(\vec p,\vec q)$} and define the set of \emph{bichromatic induced lines}
$\mathcal{B}(\vec p,\vec q)$  to be the subset of $\mathcal{L}(\vec p,\vec q)$
that is incident on at least one point of each color.
\begin{theorem}[\mainthm][\textbf{Beck-type theorem for bichromatic point sets}]\theolab{main}
Let $(\vec p,\vec q)$ be a bichromatic point set with $n$ points in each color class (for a total of $2n$).
If the maximum richness of any line in $\mathcal{L}(\vec p,\vec q)$ is $r$,
then $\card{\mathcal{B}(\vec p,\vec q)}\gg n(2n-r)$.
\end{theorem}
In the particular case where $r=n$, which is required by Gutkin and Rams, this shows that
$\card{\mathcal{B}(\vec p,\vec q)}\gg n^2$.

Beck's Theorem \ref{theo:beck}, and the present \theoref{main},
may be deduced from the famous Szemeredi-Trotter Theorem on
point-line incidences (Beck himself uses a weaker, but similar, statement as his key lemma).
The following form is what we require in the sequel.
\begin{theorem}[\stthm][\textbf{Szemeredi-Trotter Theorem \cite{ST83}}]\theolab{sttheorem}
Let $\vec p$ be a set of $n$ points in $\mathbb{E}^2$ and let $\mathcal{L}$
be a finite set of lines in $\mathbb{E}^2$.  Then the number $r$ of $k$-rich lines
in $\mathcal{L}$ is $r\ll n^2/k^3 + n/k$.
\end{theorem}

\subsection*{Notations.}  We use $\vec p = (\vec p_i)_1^n$ and $\vec q = (\vec q_i)_1^n$
for point sets in $\mathbb{E}^2$.  The notation $f(n)\gg Cg(n)$ means there is an
absolute constant $C>0$ such that $f(n) \ge g(n)$ for all $n\in \mathbb{N}$.

\section{Proofs}
The proof of the main theorem follows a similar line to Beck's original proof.  For a
pair of points $(\vec p_i,\vec q_j)$, we define the richness of the pair to be
the richness of the line $\vec p_i\vec q_j$.

\begin{lemma}\lemlab{badpairs}
Let $(\vec p,\vec q)$ be a bichromatic point set in $\mathbb{E}^2$.  Then
there is an absolute constant $K_1>0$ such that the number of bichromatic
point pairs that are either at most $1/K_1$-rich or at least $K_1n$-rich is
at least $n^2/2$.
\end{lemma}
\begin{proof}
There are exactly $n^2$ pairs of points $(\vec p_i,\vec q_j)$, and each
of these induces a line in $\mathcal{B}{(\vec p,\vec q)}$.  Define the
subset $\mathcal{B}_j{(\vec p,\vec q)}\subset \mathcal{B}{(\vec p,\vec q)}$
to be the set of bichromatic lines with richness between $2^{j-1}$ and $2^{j}$.

By the Szemeredi-Trotter Theorem with $k=2^{j}$,
\begin{equation}\eqlab{sizebj}
\card{\mathcal{B}_j{(\vec p,\vec q)}} \le C(n^2/2^{3j} + n/2^{j})
\end{equation}
The number of bichromatic pairs inducing any line $\ell\in \mathcal{B}_j{(\vec p,\vec q)}$ is maximized when there are
$2^j$ red points and $2^j$ blue ones on $\ell$, for a
total of $2^{2j}$ bichromatic pairs.  Multiplying by the estimate of \eqref{sizebj},
the number of bichromatic pairs inducing lines in $\mathcal{B}_j{(\vec p,\vec q)}$ is
at most
\begin{equation}\eqlab{pairsbj}
C(n^2/{2^j} + n2^j)
\end{equation}
for a large absolute constant $C$ coming from the Szemeredi-Trotter Theorem.

Now let $K_1>0$ be a small constant to be selected later.  We sum \eqref{pairsbj} over
$j$ such that $1/K_1\le j\le K_1{n}$:
\[
C\left(n^2\sum_{1/K_1\le j\le K_1{n}}2^{-j} +
n \sum_{1/K_1\le j\le K_1{n}} 2^{j}\right) \le
C\left(n^2{2^{-1/K_1}} + n\cdot {2K_1{n}}\right)
\]
Picking $K_1$ small enough (it depends on $C$) ensures that at most $n^2/2$ of the monochromatic pairs induce lines with richness between $1/K_1$ and $K_1 n$.
\end{proof}

The following lemma is a bichromatic variant of Beck's Two-Extremes Theorem.
\begin{lemma}\lemlab{extremes}
Let $(\vec p,\vec q)$ be a bichromatic point set in $\mathbb{E}^2$.  Then
either
\begin{itemize}
\item[\textbf{A}] The number of bichromatic lines $\card{\mathcal{B}{(\vec p,\vec q)}}\gg n^2$.
\item[\textbf{B}] There  is a line $\ell\in \mathcal{B}{(\vec p,\vec q)}$ incident on at least $K_2 n$ red points
and $K_2 n$ blue points for an absolute constant $K_2>0$.
\end{itemize}
\end{lemma}
\begin{proof}
We partition the bichromatic pairs $(\vec p_i,\vec q_j)$ into three sets: $L$ is the
set of pairs with richness less than $1/K_1$; $M$ is the set of pairs with richness
in the interval $[1/K_1,K_1 n]$; $H$ is the set of pairs with richness greater
than $K_1 n$.

By \lemref{badpairs}, $\card{L\cup H}\ge n^2/2$.  There are now three cases:

{\bf Case I: (Alternative A)} If $\card{L}\ge n^2/4$, then we
are in alternative \textbf{A}, since quadratically many pairs can be covered only by quadratically many lines of constant richness.

{\bf Case II: (Alternative B)}
If we are not in Case I, then, $\card{H}\ge n^2/4$.  In particular, since $H$ is not
empty, there are lines incident on at least one point of each color and
at least $K_1 n$ points in total.  If one of these lines is line incident to at least
$\frac{K_1}{6}n$ points of each color, then we are in alternative \textbf{B}.

\textbf{Case III: (Alternative A)}
If we are not in Case I or Case II, then every line induced by a
bichromatic pair in $H$ has at least $\frac{5}{6}K_1{n}$ red points
incident on it or $\frac{5}{6}K_1{n}$ blue ones incident on it.

Since there are $\card{H}\ge n^2/4$ bichromatic point pairs incident on a
very rich line, there must be at least $n/4$ different points of each color
participating in some point pair in $H$.

Each line induced by a pair in $H$ generates at least $K_1{n}$ incidences,
so the number of these lines is at most $\frac{1}{K_1}n$.  But then if all the
lines induced by $H$ span at most $\frac{1}{6}K_1{n}$ blue points, the total number of
blue incidences is less than $n/4$, which is a contradiction.  We can make a similar
argument for red points.

Thus there is a line $\ell_1$ spanning at least $\frac{5}{6}K_1{n}$ blue
points and a distinct line $\ell_2$ spanning at least $\frac{5}{6}K_1{n}$
red points.  From this configuration we get at least $(\frac{5}{6}K_1{n}-1)^2$
distinct bichromatic lines, putting us again in alternative \textbf{A}.
\end{proof}

\begin{proof}[Proof of \theoref{main}]
If alternative \textbf{A} of \lemref{extremes} holds, then we are already done.

If we are in alternative \textbf{B}, then there must be a line
$\ell$ of richness $r\ge 2 K_2{n}$ incident to at least $K_2{n}$
points of each color.  Now pick any subset $X$
of $K_2(2n-r)$ points not incident to $\ell$. There are at least
$\frac{1}{2}K^2_2{n}(2n-r)$ bichromatic point pairs
determined by one point in $X$ and one point incident to $\ell$.  Thus
we get at least
\[
\frac{1}{2}K^2_2{n}(2n-r) - \binom{K_2(2n-r)}{2} \ge
\frac{1}{2} K^3_2{n}(2n-r)
\]
bichromatic lines.
\end{proof}

\section{Conclusion}
We proved an extension of Beck's Theorem \cite{B83} to bichromatic point sets using a
fairly standard argument, completing the combinatorial step in Gutkin-Rams's recent
paper on billiards.

This kind of bichromatic result can, due to the general nature of the proofs, be extended
to any setting where a Szemeredi-Trotter-type result is available (see, e.g., \cite{S97}
for many examples).  Moreover, by ``forgetting'' colors and repeatedly squaring the constants,
\theoref{main} holds for bichromatic lines in multi-chromatic point sets.  It would, however,
be interesting to know whether this is the correct order of growth for the constants
in a multi-chromatic version of \theoref{main}.


\begin{thebibliography}{4}
\providecommand{\natexlab}[1]{#1}
\providecommand{\url}[1]{\texttt{#1}}
\newcommand{\doi}[1]{{doi: \href{http://dx.doi.org/#1}{\texttt{#1}}}}

\bibitem[Beck(1983)]{B83}
J{\'o}zsef Beck.
\newblock On the lattice property of the plane and some problems of {D}irac,
  {M}otzkin and {E}rd{\H o}s in combinatorial geometry.
\newblock \emph{Combinatorica}, 3\penalty0 (3-4):\penalty0 281--297, 1983.
\newblock ISSN 0209-9683.
\newblock \doi{10.1007/BF02579184}.

\bibitem[Gutkin and Rams(2010)]{GR10}
Eugene Gutkin and Michael Rams.
\newblock Complexities for cartesian products of homogeneous {L}agrangian
  systems, with applications to billiards.
\newblock Preprint, 2010.

\bibitem[Sz{\'e}kely(1997)]{S97}
L{\'a}szl{\'o}~A. Sz{\'e}kely.
\newblock Crossing numbers and hard {E}rd{\H o}s problems in discrete geometry.
\newblock \emph{Combin. Probab. Comput.}, 6\penalty0 (3):\penalty0 353--358,
  1997.
\newblock ISSN 0963-5483.
\newblock \doi{10.1017/S0963548397002976}.

\bibitem[Szemer{\'e}di and Trotter(1983)]{ST83}
Endre Szemer{\'e}di and William~T. Trotter, Jr.
\newblock Extremal problems in discrete geometry.
\newblock \emph{Combinatorica}, 3\penalty0 (3-4):\penalty0 381--392, 1983.
\newblock ISSN 0209-9683.
\newblock \doi{10.1007/BF02579194}.

\end{thebibliography}
\end{document}